\documentclass[a4paper,11pt]{article}
\usepackage{mathrsfs}
\usepackage{fancyhdr,graphicx}
\usepackage{epic}
\usepackage{epstopdf}
\usepackage{amsfonts}
\usepackage{amssymb}
\usepackage{amsthm}
\usepackage{newlfont}
\usepackage{amsmath}
\usepackage{color}
\usepackage[numbers,sort&compress]{natbib}
\numberwithin{equation}{section}
\usepackage[top=2.5cm,bottom=2.5cm,left=2.5cm,right=2.5cm]{geometry}

\newtheorem{thm}{Theorem}[section]
\newtheorem{lemma}[thm]{Lemma}

\newtheorem{con}[thm]{Conjecture}

\newtheorem{pro}[thm]{Proposition}
\newtheorem{defn}[thm]{Definition}

\newtheorem{prob}[thm]{Problem}
\newtheorem{obs}[thm]{Observation}

\usepackage{latexsym,bm}
\title{\vspace{-9mm}{The Existence of Graph whose Vertex Set Can be Partitioned into a Fixed Number of Domination Strong Critical Vertex-sets\thanks{This
work is supported by NSFC (12061047);  Undergraduate Innovation Training Project of Hubei Province (2021Byb004); Foundation of Cultivation of Scientific Institutions of Jianghan University (06210033).}}}

\author{Weisheng Zhao$^{1,2,}$\footnote{Corresponding author.}~, Ying Li$^{1}$,  Ruizhi Lin$^{3}$
\\\footnotesize{1. School of Artificial Intelligence, Jianghan University,}
\vspace{-0.2cm}
\\\footnotesize{Wuhan, Hubei 430056, China}
\vspace{-0.10cm}
\\\footnotesize{2. Institute for Interdisciplinary Research, Jianghan University,}
\vspace{-0.2cm}
\\\footnotesize{Wuhan, Hubei 430056, China}
\vspace{-0.10cm}
\\\footnotesize{3. School of computer science and mathematics, Fujian University}
\vspace{-0.2cm}
\\\footnotesize{of Technology, Fuzhou, Fujian, 350118, China}
\vspace{-0.05cm}
\\\small{E-mail addresses: weishengzhao101@aliyun.com,  liying1838@163.com,}
\vspace{-0.2cm}
\\\small{~linruizhi2018@fjut.edu.cn~~~~~~~~~~~~~}}

\date{}

\begin{document}

\maketitle

\vspace{-0.8cm}
\begin{abstract}
Let $\gamma(G)$ denote the domination number of a graph $G$.
A vertex $v\in V(G)$ is called a \emph{critical vertex} of $G$ if $\gamma(G-v)=\gamma(G)-1$. A graph is called \emph{vertex-critical} if every vertex of it is critical. In this paper, we correspondingly introduce two such definitions: (i) a set $S\subseteq V(G)$ is called a \emph{strong critical vertex-set} of $G$ if $\gamma(G-S)=\gamma(G)-|S|$; (ii) a graph $G$ is called \emph{strong $l$-vertex-sets-critical} if $V(G)$ can be partitioned into $l$  strong critical vertex-sets of $G$. Whereafter, we give some properties of strong $l$-vertex-sets-critical graphs by extending  the previous  results of vertex-critical graphs. As the core work, we study on the existence of this class of graphs and obtain that there exists a strong $l$-vertex-sets-critical connected graph if and only if $l\notin\{2,3,5\}$.
\vspace{0.0cm}
\end{abstract}



\section{Introduction}

\vspace{-1mm}
The graphs considered in this paper are finite, undirected and simple. Let $G$ be a graph with vertex set $V(G)$ and edge set $E(G)$. For any $X\subseteq V(G)$, denote by $G[X]$ the subgraph of $G$ induced by $X$. For any $v\in V(G)$, let $d_G(v)$, $N_{G}(v)$ and $N_{G}[v]$ denote the degree, open  and closed neighborhood  of vertex $v$ in $G$, respectively. Furthermore, for any $U\subseteq V(G)$, the open and closed neighbourhood of $U$ are defined as $N_G(U)=\bigcup\limits_{v\in U}N_G(v)$ and $N_G[U]=N_G(U)\cup U$, respectively. Two graphs are \emph{disjoint} if they have no common vertices.
The \emph{union} of graphs $G$ and $H$, denoted by $G\cup H$, is a graph with $V(G\cup H)=V(G)\cup V(H)$ and $E(G\cup H)=E(G)\cup E(H)$.

 A set $M \subseteq V(G)$ is called a {\it $2$-packing} of graph $G$
if $d_{G}(x,y)>2$ for every pair of distinct vertices $x, y \in M$. A  set $D\subseteq V(G)$ is called a \emph{dominating set} of $G$ if every vertex of $G$ is either in $D$ or adjacent to a vertex of $D$.  The \emph{domination number} $\gamma(G)$ is the cardinality of a minimum dominating set of $G$. We denote by $\underline{MDS}(G)$ the set of all the minimum dominating sets. That is, $\underline{MDS}(G)=\{D \mid D$ is a minimum dominating set of $G\}$.

\subsection{Domination vertex-critical}

\begin{defn}\label{is-v}
\emph{A  vertex $v\in V(G)$ is called a}  critical vertex \emph{of $G$ if} $\gamma(G-v)=\gamma(G)-1$.
\end{defn}

\begin{obs}
For any $v\in V(G)$,

\vspace{-15pt}
\begin{equation}
\gamma(G-v)=\gamma(G)-1 \Leftrightarrow \gamma(G-v)\leq\gamma(G)-1.
\label{1.1}
\end{equation}
\end{obs}

\vspace{-1pt}
\begin{defn}\label{is-cG}
\emph{ A graph $G$ is called}   vertex-critical   \emph{if every vertex of $G$ is   critical}.
\end{defn}

The research on vertex-critical graph was early in \cite{Brigham,Brigham2}. Afterwards, authors studied on its diameter \cite{Fulman},  connectivity \cite{Ananchuen},  existence of perfect matching \cite{Ananchuen,Ananchuen2,{Kazemi}} and  factor critical property \cite{Ananchuen2,Wang,Wang2} in succession. Moreover, based on the right and the left  of Formula (\ref{1.1}),  Brigham~et~al.~\cite{Brigham3} and Phillips~et~al.~\cite{Phillips} extended the notion of  vertex-critical graphs by introducing $(\gamma,k)$-critical graphs and $(\gamma,t)$-critical graphs, respectively.

\vspace{9pt}
\noindent\textbf{Definition \ref{is-cG}$^\circ$.}~\cite{Brigham3}
 A graph $G$ is called \emph{$(\gamma,k)$-critical} if $\gamma(G-S)<\gamma(G)$ for every $S\subseteq V(G)$ with $|S|=k$.

\vspace{9pt}
\noindent\textbf{Definition \ref{is-cG}$^\diamond$.}~\cite{Phillips}
 A graph $G$ is called \emph{$(\gamma,t)$-critical} if $\gamma(G-S)=\gamma(G)-t$ for every 2-packing $S$ of $G$ with $|S|=t$.

\vspace{9pt}

\noindent
In Definition \ref{is-cG}$^\circ$, if $k=2$, then $G$ is called to be \emph{domination bicritical}. For more information of  $(\gamma,k)$-critical or domination bicritical graphs, readers are suggested to refer to \cite{Mojdeh,Furuya,Furuya2,Mojdeh2,Krzywkowski,Chena}.

\vspace{11pt}
Now, again based on the left of Formula (\ref{1.1}), we introduce the definition of strong critical vertex-set to extend the notion of critical vertex in the following Definition \ref{is-v}$^\prime$. (It is easy to get that a strong critical vertex-set of $G$ is also a $2$-packing of $G$.)  To compare Definition \ref{is-cG}, we give Definition \ref{is-cG}$^\prime$.

\vspace{9pt}
\noindent\textbf{Definition \ref{is-v}$^\prime$.}~\cite{Zhao}
A set $S\subseteq V(G)$ is called a \emph{strong critical vertex-set}  (or just  \emph{st-critical vertex-set}) of $G$ if $\gamma(G-S)=\gamma(G)-|S|$.

\vspace{10pt}

\noindent\textbf{Definition \ref{is-cG}$^\prime$.}
A graph $G$ is called \emph{strong~$l$-vertex-sets-critical} if $V(G)$ can be partitioned into $l$ (non-empty) strong critical vertex-sets of $G$.

\subsection{On strong critical vertex-set and two-colored $\gamma$-set}

When we talk about st-critical vertex-set, we would like to mention another related notion ------ two-colored $\gamma$-set. The present authors think that both of them are important on the problem of building family of graphs that make the equality hold in Vizing's Conjecture \cite{Hartnell0,Zhao}.

\begin{defn}\label{is-v4}\emph{\cite{Hartnell0}}
\emph{Let $D\in \underline{MDS}(G)$. $D$ is called a} two-colored \emph{$\gamma$-set of $G$ if $D$  partitions as $D=D_1 \cup D_2$ such that $V(G)-N_G[D_1]=D_2$ and $V(G)-N_G[D_2]=D_1$.}
\end{defn}

\noindent In Definition \ref{is-v4}, since $V(G)-N_G[D_1]=D_2$, we can deduce that $D_1\in \underline{MDS}(G-D_2)$. So $\gamma(G-D_2)=|D_1|=|D|-|D_2|=\gamma(G)-|D_2|$, which implies that $D_2$  is an st-critical vertex-set of $G$, and so is $D_1$ symmetrically. Because two-colored  $\gamma$-set is not the motif of this paper, we just introduce a proposition and a conjecture about it below, where in Proposition \ref{twc}, $``\square"$ represents the cartesian product and a nontrivial connected graph $G$ is called a \emph{generalized comb} if each vertex of degree greater than one is adjacent to exactly one 1-degree-vertex of $G$.

\begin{pro}\label{twc}\emph{\cite{Hartnell0}}
If $G$ is a generalized comb and $H$ has a two-colored $\gamma$-set, then $\gamma(G \square H)=\gamma(G)\gamma(H)$.
\end{pro}

\begin{con}\label{twc2}\emph{\cite{Hartnell4}}
If $G$ is a connected bipartite graph such that $V(G)$ can be partitioned into two-colored $\gamma$-sets, then $G$ is the $4$-cycle or $G$ can be obtained from $K_{2t,2t}$ by removing the edges of $t$ vertex-disjoint $4$-cycles.
\end{con}

\medskip

At last, we account for the coming two sections. We will compare the properties of vertex-critical graphs and strong $l$-vertex-sets-critical graphs in Section 2, for the reason that
st-critical vertex-set is a generalization of critical vertex as well as strong $l$-vertex-sets-critical graphs is a special kind of vertex-critical graphs. Let $\mathscr{C}_l=\{G\mid G$ is a strong $l$-vertex-sets-critical connected graph$\}$. We will obtain that $\mathscr{C}_l\neq\emptyset$ if and only if $l\notin\{2,3,5\}$ in Section 3.

\section{To compare vertex-critical and strong $l$-vertex-sets-critical}
%
Brigham~et~al.~\cite{Brigham2} studied on the vertex-critical graphs, and listed the following Theorems \ref{ab} and \ref{cc-c} without proofs because they thought the proofs are cumbersome but straightforward. In order to state these two theorems, we have to introduce the notion of vertex coalescence first.

\begin{defn}\label{vc}\emph{\cite{Brigham2,Haynes2}}
\emph{Let $G$ and $H$ be two disjoint graphs with $g\in V(G)$ and $h\in V(H)$. The} vertex coalescence \emph{$G\cdot_{gh}H$ (or just $G\cdot H$ if $g$ and $h$ are arbitrary) of $G$ and $H$ via $g$ and $h$, is the graph obtained from the  union of $G$ and $H$ by identifying the vertices $g$ and $h$ as one vertex $g$.} \emph{(Note: 1.~An example of vertex coalescence is shown in Figure \ref{gabh};~~2.~By the way, for the edge coalescence, readers can refer to \cite{Grobler}.)}
\end{defn}

\vspace{1mm}
\begin{figure}[h]
\centering
\begin{minipage}[t]{0.83\linewidth}
\centering
\includegraphics[width=0.82\textwidth]{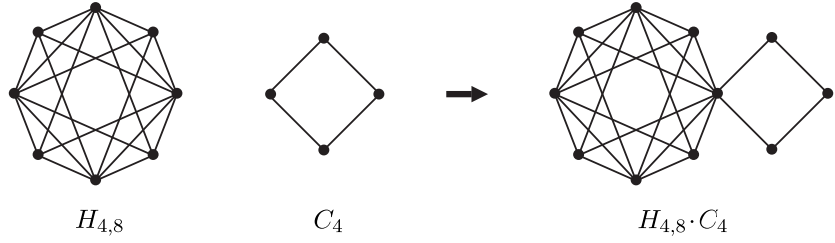}
\vspace{-2mm} \caption{The vertex coalescence of graphs $H_{4,8}$ and $C_4$.}\label{gabh}
\end{minipage}
\end{figure}

\vspace{-2mm}
\begin{thm}\label{ab}\emph{\cite{Brigham2}}
Let $G$ and $H$ be two disjoint graphs. Form any vertex coalescence $G\cdot H$. Then $\gamma(G)+\gamma(H)-1\leq\gamma(G\cdot H)\leq\gamma(G)+\gamma(H)$. Furthermore, if both $G$ and $H$ are vertex-critical or $G\cdot H$ is vertex-critical, then $\gamma(G\cdot H)=\gamma(G)+\gamma(H)-1$.
\end{thm}

\begin{thm}\label{cc-c}\emph{\cite{Brigham2}}
The graph $G\cdot H$ is vertex-critical if and only if both $G$ and $H$  are vertex-critical.
\end{thm}


\vspace{1mm}
To compare Brigham's results, we give the corresponding results on strong $l$-vertex-sets-critical graphs one to one (see Definition \ref{vc}$^\prime$,  Theorems \ref{ab}$^\prime$ and \ref{cc-c}$^\prime$). For the mathematical rigor, we are going to prove them without the supporting of Theorems \ref{ab} and \ref{cc-c}, where in fact, our proofs include the derivation of Brigham's results.  Before this, we need to display four observations, two definitions and one lemma.
\begin{obs}\label{G1G2}
Let $G$ be a graph. If $G_1$ and $G_2$ are vertex-induced subgraphs of $G$ such that $V(G)=V(G_1)\cup V(G_2)$. Then $\gamma(G)\leq \gamma(G_1)+\gamma(G_2)$ with  the equality holds if $G_1$ and $G_2$ are two components of $G$.
\end{obs}

\begin{obs}\label{GS2}
\emph{(a)} For any $S\subseteq V(G)$, $\gamma(G-S)\leq\gamma(G)-|S|$ $\Leftrightarrow$ $\gamma(G-S)=\gamma(G)-|S|$.

~~~~~~~~~~~~~~~~~~~~\emph{(b)} For any $S\subseteq V(G)$,  $\gamma(G-S)\geq\gamma(G)-|S|$.
\end{obs}

\begin{obs}\label{cc-s}\emph{\cite{Zhao}}
A subset of an st-critical vertex-set of $G$ is still an st-critical vertex-set of $G$.
\end{obs}

\begin{obs}\label{S1S2}
Let $S$ be an st-critical vertex-set of $G$, and $S_1,S_2\subseteq S$ with $S_1\cap S_2=\emptyset$. Then $S_1$ is an st-critical vertex-set of $G-S_2$.
\end{obs}

\begin{defn}
\emph{Let $S_1,S_2,\ldots,S_l$ be non-empty strong critical vertex-sets of $G$. If} $\{S_1,S_2, \ldots,$ $S_l\}$ \emph{is a partition of $V(G)$, then we call $\{S_1,S_2,\ldots,S_l\}$ or ${S_1\cup S_2\cup\cdots\cup S_l}$  as a} \emph{strong critical vertex-sets partition} \emph{of $G$}.
\end{defn}


\begin{defn}
\emph{Let $J$ be a graph with $x,y\in V(J)$. $x,y$ are called} mutually compatible \emph{in $J$ if there exists $D_0\in \underline{MDS}(J)$ such that $\{x,y\}\subseteq D_0$, and} mutually incompatible \emph{in $J$} \emph{if $|\{x,y\}\cap D|<2$ for any $D\in \underline{MDS}(J)$.}
\end{defn}

\begin{lemma}\label{Jxy}
Let $J$ be a graph with $x,y\in V(J)$, and $J^{\prime}$ be the graph obtained from $J$ by identifying the two vertices $x$ and $y$ as one vertex $x$. Then $\gamma(J)-1\leq \gamma(J^{\prime})\leq \gamma(J)$ with the second equality holds if and only if $x,y$ are mutually incompatible and neither   $x$ nor $y$ is critical in $J$.
\end{lemma}
\begin{proof}
Let $D^{\prime}\in \underline{MDS}(J^{\prime})$. Then $D^{\prime}\cup\{y\}$ is a dominating set of $J$, and so $\gamma(J)\leq|D^{\prime}\cup\{y\}|\leq\gamma(J^{\prime})+1$. Let $D\in \underline{MDS}(J)$ and
\begin{equation*}
D_0^{\prime}=
\begin{cases}
D, &~\text{if~$y\notin D;$}\\
(D-\{y\})\cup\{x\}, &~\text{if~$y\in D$.}
\end{cases}
\end{equation*}
Then $D_0^{\prime}$ is a dominating set of $J^{\prime}$, and so $\gamma(J^{\prime})\leq|D_0^{\prime}|\leq|D|=\gamma(J)$.

\smallskip
Furthermore, suppose that $x,y$ are mutually incompatible and neither of $x$ nor $y$ is critical in $J$. We need to prove that $\gamma(J^{\prime})=\gamma(J)$. Let $D^{\prime}\in \underline{MDS}(J^{\prime})$. There are two cases. If $x\in D^{\prime}$, then $D^{\prime}\cup \{y\}$ is a dominating set of $J$. Since $x,y$ are mutually incompatible in $J$, we have $\gamma(J^{\prime})+1=|D^{\prime}\cup \{y\}|\geq\gamma(J)+1$, which implies that $\gamma(J^{\prime})=\gamma(J)$. If $x\notin D^{\prime}$, then $N_{J^{\prime}}(x)\cap D^{\prime}\neq\emptyset$. So one of $x$ and $y$, say $x$, can be dominated by $D^{\prime}$ in $J$. Thus $D^{\prime}$ is a dominating set of $J-y$. Since   $y$ is not critical in $J$,  we have $\gamma(J)\leq\gamma(J-y)\leq|D^{\prime}|=\gamma(J^{\prime})$, which also implies that $\gamma(J^{\prime})=\gamma(J)$.

Conversely, if $\gamma(J^{\prime})=\gamma(J)$, we prove firstly that $x,y$ are mutually incompatible in $J$. Otherwise, let $D_{xy}\in\underline{MDS}(J)$ with $\{x,y\}\subseteq D_{xy}$. Then $D_{xy}-\{y\}$ is a dominating set of $J^{\prime}$ with cardinality $\gamma(J)-1=\gamma(J^{\prime})-1<\gamma(J^{\prime})$, a contradiction. We prove secondly that neither of $x$ nor $y$ is critical in $J$. Otherwise, we have that one of $x$ and $y$, say $x$, is critical in $J$.  Let $D^{-}\in\underline{MDS}(J-x)$ and $D_x=D^{-}\cup\{x\}$. Then $D_x\in\underline{MDS}(J)$. Since $x,y$ are mutually incompatible in $J$, we have $y\notin D_x$. So $N_{J}(y)\cap D_x\neq\emptyset$, which implies that $N_{J^{\prime}}(x)\cap D_x\neq\emptyset$. Thus $D_x-\{x\}$ is a dominating set of $J^{\prime}$. Hence $\gamma(J^{\prime})\leq|D_x-\{x\}|=\gamma(J)-1$, also a contradiction.
\end{proof}

\smallskip

\noindent\textbf{Definition \ref{vc}$^\prime$.}
Let $G$ and $H$ be two disjoint graphs with $\emptyset\neq X\subseteq V(G)$, $\emptyset\neq Y\subseteq V(H)$ and $|X|=|Y|$. Let $X=\{x_1,x_2,\ldots,x_m\}$ and $Y=\{y_1,y_2,\ldots,y_m\}$. The \emph{vertex-set coalescence} $G\cdot_{XY}H$  of $G$ and $H$ via $X$ and $Y$, is the graph obtained from the  union of $G$ and $H$ by identifying the vertices $x_i$ and $y_i$ as one vertex $x_i$ for every $1\leq i\leq m$.
(Refer to Figure \ref{GXYH}.)

\vspace{11pt}

\vspace{0mm}
\begin{figure}[h]
\centering
\begin{minipage}[t]{0.88\linewidth}
\centering
\includegraphics[width=0.87\textwidth]{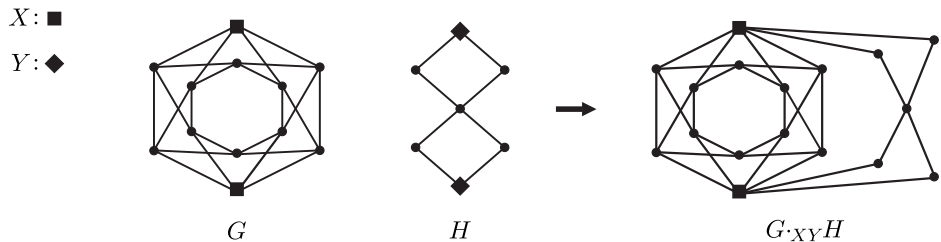}
\vspace{-2mm} \caption{Illustration for the vertex-set coalescence $G\cdot_{XY}H$.}\label{GXYH}
\end{minipage}
\end{figure}

\newpage
\vspace{0pt}
\noindent\textbf{Theorem \ref{ab}$^\prime$.}
\emph{Let $G$ and $H$ be two disjoint graphs with $\emptyset\neq X\subseteq V(G)$, $\emptyset\neq Y\subseteq V(H)$ and $|X|=|Y|$. Then}

(a) $\gamma(G)+\gamma(H)-|X|\leq\gamma(G\cdot_{XY}\negthinspace H)\leq\gamma(G)+\gamma(H)$;

(b) \emph{$X$ and $Y$ are st-critical vertex-sets of $G$ and $H$ respectively if and only if $X$ is an st-critical vertex-set of $G\cdot_{XY}H$;}

(c) \emph{if $X$ is an st-critical vertex-set of $G\cdot_{XY}H$, then $\gamma(G\cdot_{XY}\negthinspace H)=\gamma(G)+\gamma(H)-|X|$.}

\begin{proof}
\textbf{(a)} Firstly, let $D^{\prime}\in \underline{MDS}(G\cdot_{XY}H)$. Then $D^{\prime}\cup Y$ is a dominating set of $G\cup H$, and so $\gamma(G)+\gamma(H)=\gamma(G\cup H)\leq|D^{\prime}\cup Y|=\gamma(G\cdot_{XY}H)+|Y|$, which implies that $\gamma(G)+\gamma(H)-|X|\leq\gamma(G\cdot_{XY}\negthinspace H)$. Secondly, let $H^{\prime}$ be the subgraph of $G\cdot_{XY}\negthinspace H$ induced by $V(H-Y)\cup X$. Then $H^{\prime}\cong H$. Let $D_G\in\underline{MDS}(G)$ and $D_{H^{\prime}}\in\underline{MDS}(H^{\prime})$. Then $D_G\cup D_{H^{\prime}}$ is a dominating set of $G\cdot_{XY}\negmedspace H$, and so $\gamma(G\cdot_{XY}\negmedspace H)\leq|D_G\cup D_{H^{\prime}}|\leq\gamma(G)+\gamma(H)$.

\smallskip
\textbf{(b)} ($\Rightarrow$) Let $D_G^{-}\in \underline{MDS}(G-X)$ and $D_H^{-}\in \underline{MDS}(G-Y)$. Then $D_G^{-} \cup D_H^{-}$ is a dominating set of $G\cdot_{XY}H-X$. So $\gamma(G\cdot_{XY}H-X)\leq|D_G^{-} \cup D_H^{-}|=\gamma(G)-|X|+\gamma(H)-|Y|=\gamma(G)+\gamma(H)-2|X|$. By Item (a), we have $\gamma(G)+\gamma(H)-2|X|\leq\gamma(G\cdot_{XY}H)-|X|$. Thus $X$ is an st-critical vertex-set of $G\cdot_{XY}\negthinspace H$.

($\Leftarrow$)
We are going to prove the sufficiency by induction on $|X|$. When $|X|=1$, we let $X=\{x\}$, $Y=\{y\}$, and $J=G\cup H$. If $\gamma(G\cdot_{xy}\negmedspace H)=\gamma(G)+\gamma(H)$,  then by Lemma \ref{Jxy}, neither $x$ nor $y$ is critical in $J$. So $\gamma(G)+\gamma(H)-1=\gamma(G\cdot_{xy}H)-1=\gamma(G\cdot_{xy}H-x)=\gamma(G-x)+\gamma(H-y)\geq\gamma(G)+\gamma(H)$, a contradiction. Thus we have $\gamma(G\cdot_{xy}H)=\gamma(G)+\gamma(H)-1$ by Item (a). So $\gamma(G)-1+\gamma(H)-1=\gamma(G\cdot_{xy}H)-1=\gamma(G\cdot_{xy}H-x)=\gamma(G-x)+\gamma(H-y)\geq\gamma(G)-1+\gamma(H)-1$, from which we get $\gamma(G-x)=\gamma(G)-1$ and $\gamma(H-y)=\gamma(H)-1$, and so the sufficiency holds.

Suppose that the sufficiency holds when $|X|=n$ ($n\geq1$). We consider the case when $|X|=n+1$ below. Let $x \in X$, $y\in Y$, $X_0=X-\{x\}$, $Y_0=Y-\{y\}$, $J=G\cdot_{X_0Y_0}\negmedspace H$ and $J^{\prime}=G\cdot_{XY}\negmedspace H$. Let $D_1\in \underline{MDS}(G-X)$, $D_2\in \underline{MDS}(H-Y)$ and $D=D_1\cup X \cup D_2$. Since $X$ is an st-critical vertex-set of $J^{\prime}$, it follows that $D\in \underline{MDS}(J^{\prime})$. Also, $D$ is a dominating set of $J-y$. So $\gamma(J-y)\leq|D|=\gamma(J^{\prime})$.  By Definition \ref{is-v} and Lemma \ref{Jxy}, we have
\begin{equation*}
\gamma(J-y)
\begin{cases}
=\gamma(J)-1=\gamma(J^{\prime}), &~\text{if~$y$ is a critical vertex of $J$;}\\
\geq\gamma(J)\geq\gamma(J^{\prime}), &~\text{if~$y$ is not a critical vertex of $J$,}
\end{cases}
\end{equation*}
 from which we know $\gamma(J-y)\geq\gamma(J^{\prime})$. Thus $\gamma(J-y)=\gamma(J^{\prime})$. Therefore $\gamma\big((J-y)-X\big)=\gamma(G-X)+\gamma(H-Y)=\gamma(J^{\prime}-X)=\gamma(J^{\prime})-|X|=\gamma(J-y)-|X|$, which implies that $X$ is, and so $X_0$ is, an st-critical vertex-set of $J-y$. By Observation \ref{S1S2}, we see that $x$ is a   critical vertex of $J^{\prime}-X_0$. Note that $J-y=G\cdot_{X_0Y_0}(H-y)$ and $J^{\prime}-X_0=(G-X_0)\cdot_{xy}(H-Y_0)$.  By the inductive hypothesis, we have $X_0$ is an st-critical vertex-set of $G$ as well as $x$ is critical in $G-X_0$. Hence $\gamma(G-X)=\gamma\big((G-X_0)-x\big)=\gamma(G-X_0)-1=\gamma(G)-|X_0|-1=\gamma(G)-|X|$. That is to say, $X$ is an st-critical vertex-set of $G$. Symmetrically, one can prove that  $Y$ is an st-critical vertex-set of $H$.
Thus the result is true when $|X|=n+1$. Item (b) follows.

\smallskip
\textbf{(c)} Let $D^{\circ}\in\underline{MDS}(G\cdot_{XY}\negthinspace H-X)$. Then by Item (b), we have $\gamma(G\cdot_{XY}\negthinspace H)=|D^{\circ}|+|X|=\gamma(G-X)+\gamma(H-Y)+|X|=\gamma(G)-|X|+\gamma(H)-|Y|+|X|=\gamma(G)+\gamma(H)-|X|$.
\end{proof}

\noindent\textbf{Theorem \ref{cc-c}$^\prime$.}
\emph{Let $G$ and $H$ be two disjoint graphs.} \emph{Let $\emptyset\neq X_i\subseteq V(G)$ for $i=1,2,\ldots,k$ and $\emptyset\neq Y_j\subseteq V(H)$ for $j=1,2,\ldots,l$ with $|X_1|=|Y_1|$. Then
 $\{X_1, X_2,\ldots, X_k\}$ and $\{Y_1, Y_2,\cdots, Y_l\}$ are partitions of st-critical vertex-sets of $G$ and $H$ respectively if and only if $\{X_1, X_2,\ldots, X_k,Y_2,$ $Y_3,\ldots, Y_l\}$ is a partition of st-critical vertex-sets of $G._{X_1Y_1}H$.}
\begin{proof}
Let $\mathbb{X}=\{X_1, X_2,\ldots, X_k\}$, $\mathbb{Y}=\{Y_1, Y_2,\cdots, Y_l\}$ and $\mathbb{X}.\mathbb{Y}=\{X_1, X_2,\ldots, X_k,Y_2, Y_3,$ $\ldots, Y_l\}$.

$(\Rightarrow)$ Clearly, $\mathbb{X}.\mathbb{Y}$ is a partition of $V(G\cdot_{X_1Y_1}\negmedspace H)$. For any $S\in \mathbb{X}.\mathbb{Y}$, we have $S\in \mathbb{X}$ or $S\in\mathbb{Y}$. If $S\in \mathbb{X}$, then by Theorem \ref{ab}$^\prime$ (c), we have
 $\gamma(G\cdot_{X_1Y_1}\negmedspace{H}-S)\leq\gamma(G-S)+\gamma(H-Y_1)=\gamma(G)-|S|+\gamma(H)-|X_1|=\gamma(G\cdot_{X_1Y_1}H)-|S|$. Similarly, we can also prove that $\gamma(G\cdot_{X_1Y_1}H-S)\leq \gamma(G\cdot_{X_1Y_1}H)-|S|$ if $S\in \mathbb{Y}$. The necessity follows.

$(\Leftarrow)$ Clearly,  $\mathbb{X}$ and $\mathbb{Y}$ are partitions of $V(G)$ and $V(H)$, respectively. Firstly, by Theorem \ref{ab}$^\prime$ (b), $X_1$ and $Y_1$ are st-critical vertex-sets of $G$ and $H$, respectively. Secondly, for any $S\in \mathbb{X}-\{X_1\}$, we let $D^{-}\in \underline{MDS}(G\cdot_{X_1Y_1}\negmedspace{H}-S)$, $L=X_1-(X_1\cap D^{-})$ and $L_G$ be the subset of $L$ that can be dominated by $D^{-}\cap V(G)$. Let $H^{\prime}$ be the subgraph of $G\cdot_{X_1Y_1}\negmedspace H$ induced by $V(H-Y_1)\cup X_1$. Then $D^{-}\cap V(G)$ and $D^{-}\cap V(H^\prime)$ are dominating sets of $(G-S)-(L-L_G)$ and $H^\prime-L_G$, respectively. So
\begin{eqnarray*}
|D^{-}| & = & |D^{-}\cap V(G)|+|D^{-}\cap V(H^\prime)|-|D^{-}\cap X_1| \\
& \geq & \gamma\big((G-S)-(L-L_G)\big)+\gamma(H^\prime-L_G)-|X_1\cap D^{-}|  \\
& \geq & \gamma(G-S)-|L-L_G|+\gamma(H^\prime)-|L_G|-|X_1\cap D^{-}|~~~(\text{By~Observation \ref{GS2} (b)})\\
& \geq & \gamma(G)-|S|+\gamma(H)-|X_1| \\
& = & \gamma(G\cdot_{X_1Y_1}\negmedspace H)-|S|~~~(\text{By~Theorem \ref{ab}$^\prime$ (c)})\\
& = & |D^{-}|.
\end{eqnarray*}
By the forth equality, we have $\gamma(G-S)=\gamma(G)-|S|$. Thirdly, for any $S\in \mathbb{Y}-\{Y_1\}$, we can similarly prove that $\gamma(H-S)=\gamma(H)-|S|$. From these three observations,  the sufficiency follows, too.
\end{proof}

\vspace{-4mm}
\section{The existence}

\vspace{-2mm}
In this section, we write $d_G(\cdot)=d(\cdot)$, $N_G(\cdot)=N(\cdot)$ and $N_G[\cdot]=N[\cdot]$, as well as $C_4\cdot C_4=(C_4)^2$, $C_4\cdot C_4\cdot C_4=(C_4)^3$ and so on for belief.

\begin{lemma}\label{cc-f}\emph{\cite{Zhao}}
An st-critical vertex-set of a graph $G$ is a $2$-packing of $G$.
\end{lemma}

\begin{lemma}\label{Th334}\emph{\cite{Wang3}}
 If $d(u)=1$ and $v\in N(u)$, then $v$ is not a  critical vertex of $G$. \emph{(This implies that a vertex-critical graph has no vertices of degree one.)}
\end{lemma}

\begin{lemma}\label{shd}\emph{\cite{Zhao}}
Let $S$ be an st-critical vertex-set of $G$. If $D^{-}_{G}\in \underline{MDS}(G-S)$, then $|D^{-}_{G}|=\gamma(G)-|S|$ and $D^{-}_{G}\cap N(S)=\emptyset$.
\end{lemma}

\begin{proof}
Firstly, from the definition of st-critical vertex-set, we have $|D^{-}_{G}|=\gamma(G)-|S|$. Secondly, if $D^{-}_{G}\cap N(S)\neq\emptyset$, let $L=N(D^{-}_{G})\cap  S$. Then $D^{-}_{G}$ is a dominating set of $G-(S-L)$, and so $\gamma(G-(S-L))\leq |D^{-}_{G}|=\gamma(G)-|S|<\gamma(G)-|S-L|$. However, we have $S-L$ is an st-critical vertex-set of $G$ by Observation \ref{cc-s}, which implies that $\gamma(G-(S-L))=\gamma(G)-|S-L|$, a contradiction.
\end{proof}

\begin{lemma}\label{y-t}
Let  $S$ be an st-critical vertex-set of $G$ and $w\in V(G-S)$.

\emph{(a)}  If $z\in N(w)\cap S$, then there exists $v_0\in N(w)-\{z\}$ such that $N(v_0)\cap S=\emptyset$.

\emph{(b)} Let $uvwz$ be a path or a cycle in $G$  $($i.e. $u=z$ is possible$)$. If $u,z\in S$, then $d(w)>2$.

\emph{(c)}  Let $X=N(w)$. If $2\leq|X|\leq3$ and $N(x)\cap S\neq\emptyset$ for every $x\in X$, then $|N(X)\cap S|=1$.

\emph{(d)}  Let $uvwyz$ be a trail in $G$. If $u,z\in S$ and $d(w)=2$, then $u=z$.
\end{lemma}
\begin{proof}
\textbf{(a)} Suppose to the contrary that $N(v)\cap S\neq\emptyset$ for every $v\in N(w)-\{z\}$. Then $N[w]-\{z\}\subseteq N(S)$. By Lemma \ref{shd},  there exists $D_G^{-}\in \underline{MDS}(G-S)$ such that  $ D_G^{-}\cap(N[w]-\{z\})=\emptyset$. But now, we see that $D_G^{-}$ can not dominate $w$ in $G-S$, a contradiction.


\textbf{(b)} It is an immediate result of Item (a).

\textbf{(c)}  Suppose to the contrary that $|N(X)\cap S|\neq1$. By Lemma \ref{cc-f}, we have $|N(X)\cap S|\leq|X|$. This implies $|N(X)\cap S|=2$ or $3$. Let $\{r,s\}\subseteq N(X)\cap S$. Then $N(r)\cap N(s)=\emptyset$. So we must have that at least one of $r$ and $s$, say $r$, is adjacent to only one element of $X$. Thus we may suppose that $\{r\}=N(x^{\prime})\cap S$, where $x^{\prime}\in X$. Note that $N(x)\cap S\neq\emptyset$ for every $x\in X$ implies $X\subseteq N(S)$. By Lemma \ref{shd}, there exists $D_G^{-}\in \underline{MDS}(G-S)$ such that $D_G^{-}\cap X=\emptyset$ and $|D_G^{-}|+|S|=\gamma(G)$. In order to dominate $w$ in $G-S$, we have $w\in D_G^{-}$. But then $(D_G^{-}-\{w\})\cup (S-\{r\})\cup\{x^{\prime}\}$ is a dominating set of $G$ with cardinality $\gamma(G)-1$, a contradiction.

\textbf{(d)} It is an immediate result of Item (c).
\end{proof}

\begin{thm}\label{main}
There exists a connected graph $G$ such that $V(G)$ can be partitioned into $l$  strong critical vertex-sets if and only if $l\notin\{2,3,5\}$.
\end{thm}
\begin{proof}
$(\Leftarrow)$ Let $k\in \mathbb{Z}^+$ and $H_{4,8}$ be the (Harary) graph as shown in Figure \ref{gabh}.~Based on the fact that $\mathbb{Z}^+-\{2,3,5\}=\{1\}\cup\{3k\mid k\geq2\}\cup\{3k+1\mid k\geq1\}\cup\{3k+2\mid k\geq2\}$, we let
\begin{equation*}
G=
\begin{cases}
K_1, & \mbox{if~$l=1$};\\
(C_4)^{k}, & \mbox{if~$l\in \{3k\mid k\geq2\}\cup\{3k+1\mid k\geq1\}$},\\
H_{4,8} \cdot (C_4)^{k-2}, & \mbox{if~$l\in \{3k+2\mid k\geq2\}$}.
\end{cases}
\end{equation*}
Noting that   $V(C_4)$ and $V(H_{4,8})$ can be partitioned into 4 and 8 st-critical vertex-sets respectively, we can recursively deduce that $V\big((C_4)^{k}\big)$ and $V\big(H_{4,8} \cdot (C_4)^{k-2}\big)$ can be partitioned into $3k+1$ and $3k+2$ ($k\geq2$) st-critical vertex-sets respectively  by Theorem \ref{cc-c}$^{\prime}$. Also, note that $V\big((C_4)^{2}\big)$

\noindent can be partitioned into 6 st-critical vertex-sets. So $V\big((C_4)^{k}\big)$ can be partitioned into $3k$ ($k\geq2$) st-critical vertex-sets. The sufficiency follows.


\smallskip
$(\Rightarrow)$ Suppose to the contrary that $l\in\{2,3,5\}$. If $l=2$, then by Lemma \ref{cc-f}, we get that $d(h)=1$ for every $h\in V(G)$, which implies that $G\cong K_2$, contradicting the fact that $K_2$ is not a vertex-critical graph. If $l=3$, then by Lemmas \ref{Th334} and \ref{cc-f}, we deduce that $d(h)=2$ for every $h\in V(G)$, which implies that $G$ is a cycle. However, one can check that this is impossible. (According to the two well-known facts that $\gamma(C_n)=\lceil\frac{n}{3}\rceil$ and $\gamma(P_n)=\lceil\frac{n}{3}\rceil$, we deduce that a cycle of order at least 4 can not own an st-critical vertex-set of cardinality $2$.)

If $l=5$, then we let $V(G)=S_1\cup S_2\cup S_3\cup S_4\cup S_5$ be a partition of st-critical vertex-sets of $G$. By Lemmas \ref{Th334} and  \ref{cc-f}, we get that $2\leq d(g)\leq 4$ for every $g\in V(G)$.

\smallskip
\emph{\textbf{Claim 1.}}~ Let $\{j,k,l,m,n\}=\{1,2,3,4,5\}$. If $N(s_n)=\{s_j,s_k,s_l\}$, where $s_i\in S_i$ for every $i\in\{j,k,l,n\}$, then $|N(\{s_j,s_k,s_l\})\cap S_m|=1$.

For convenience, suppose without loss of generality that $(j,k,l,m,n)=(1,2,3,4,5)$.  We use reduction to absurdity. Assume that $|N(\{s_1,s_2,s_3\})\cap S_4|\neq1$.
Then since

\vspace{-12pt}
\begin{equation*}
N(s_5)=\{s_1,s_2,s_3\},
\end{equation*}

 \noindent by the contrapositive of Lemma \ref{y-t} (c), at least one of $s_1$, $s_2$ and $s_3$, say $s_1$, satisfies $N(s_1)\cap S_4=\emptyset$. Thus $N(s_1)-\{s_5\}\subseteq S_2\cup S_3$. To combine this with Lemma \ref{y-t} (b), we must have $d(s_1)\neq 2$, which implies that $d(s_1)=3$. So $N(s_1)\cap S_2\neq\emptyset$ and $N(s_1)\cap S_3\neq\emptyset$.

Since $s_1\in N(s_5)\cap S_1$, by Lemma \ref{y-t} (a), one of $N(s_2)\cap S_1$ and $N(s_3)\cap S_1$, say $N(s_2)\cap S_1$, is empty set. By Lemma \ref{cc-f}, we have $(N(s_2)-\{s_5\})\cap (S_2\cup S_5)=\emptyset$.  Since $s_3\in N(s_5)$ and $N(s_1)\cap S_3\neq\emptyset$, by Lemma \ref{y-t} (a), we obtain that $N(s_2)\cap S_3=\emptyset$. So by Lemma \ref{Th334}, we get $N(s_2)\cap S_4\neq\emptyset$. Let $N(s_2)\cap S_4=\{s_4\}$. Then

\vspace{-13.5pt}
\begin{equation*}
N(s_2)=\{s_4,s_5\}.
\end{equation*}

\vspace{-1pt}
\noindent So we have $N(s_4)\cap  S_5=\emptyset$ by the contrapositive of Lemma \ref{y-t} (b). Thus $N(s_4)-\{s_2\}\subseteq S_1\cup S_3$. Since $d(s_2)=2$, we have $s_4s_3\in E(G)$ or $s_4s_1\in E(G)$ by Lemma \ref{y-t} (d). But we have supposed that $N(s_1)\cap S_4=\emptyset$ in the third sentence of the first paragraph. Thus, only $s_4s_3\in E(G)$ holds, which implies that

\vspace{-25pt}
\begin{equation*}
N(s_4)=\{s_2,s_3\}.
\end{equation*}

\vspace{-4pt}
\noindent  So by Lemma \ref{y-t} (b), we get

 \vspace{-22pt}
\begin{equation*}
N(s_3)\cap S_2=\emptyset.
\end{equation*}

\vspace{3pt}
  Now, if $N(s_3)\cap S_1=\emptyset$ (refer to Figure \ref{123} (i-a)) or $N(s_3)\cap S_1=\{s_1\}$ (refer to Figure \ref{123} (i-b)), then $s_1$ is a cut-vertex of $G$. Thus by Theorem \ref{cc-c}, $G[\{s_1,s_2,s_3,s_4,s_5\}]$ is vertex-critical. But one can check that it is not true.  So we can let $N(s_3)\cap S_1=\{r_1\}$ with $r_1\neq s_1$. However, by Lemma \ref{shd}, there exists $D_G^{-}\in \underline{MDS}(G-\{r_1,s_1\})$ such that $D_G^{-}\cap \{s_3,s_5\}=\emptyset$ and $|D_G^{-}|+2=\gamma(G)$. In order to dominate $\{s_2,s_4\}$ in $G-\{r_1,s_1\}$, we have $|\{s_2,s_4\}\cap D_G^{-}|=1$. But then $(D_G^{-}-\{s_2,s_4\})\cup \{s_3,s_5\}$ is a dominating set of $G$ with cardinality $\gamma(G)-1$, a contradiction.


\newpage

\begin{figure}[h]
\centering
\begin{minipage}[t]{0.95\linewidth}
\centering
\includegraphics[width=0.95\textwidth]{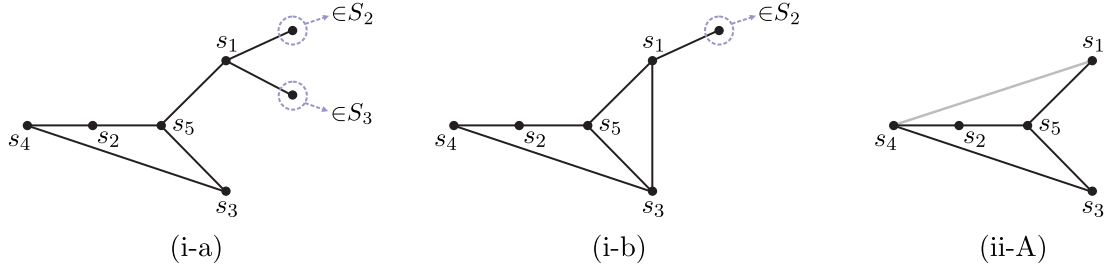}
\vspace{-2mm} \caption{Illustration for the proofs of Claim 1 and Claim 2-A}\label{123}
\end{minipage}
\end{figure}

\emph{\textbf{Claim 2.}}~ $d(g)\neq3$ for every $g\in V(G)$.

Without loss of generality, suppose to the contrary that there exists $s_5\in S_5$ such that $N(s_5)=\{s_1,s_2,s_3\}$, where $s_i\in S_i$, $i=1,2,3$. By Claim 1, we can let

\vspace{-5mm}
\begin{equation}
N(\{s_1,s_2,s_3\})\cap S_4=\{s_4\}.
\label{123-4}
\end{equation}

\vspace{1mm}
\textbf{Case A.} At least two of $s_1$, $s_2$ and $s_3$, say $s_2$ and $s_3$, have degree $2$ in $G$.

Then by  Lemma \ref{y-t} (b), we have $N(s_2)\cap S_1=\emptyset$ and $N(s_3)\cap S_1=\emptyset$, as well as  $N(s_2)\cap S_3=\emptyset$ and $N(s_3)\cap S_2=\emptyset$. So we must have $N(s_2)\cap S_4\neq\emptyset$ and $N(s_3)\cap S_4\neq\emptyset$ by Lemma \ref{Th334}. By (\ref{123-4}), we get $N(s_2)\cap S_4=\{s_4\}=N(s_3)\cap S_4$. Again by Lemma \ref{y-t} (b), we have $N(s_4)\cap S_5=\emptyset$.

If $N(s_4)\cap S_1\neq \emptyset$, then by Lemma \ref{y-t} (d), we have $N(s_4)\cap S_1=\{s_1\}$. From this, we see that either $G=G[\{s_1,s_2,s_3,s_4,s_5\}]$, or $s_1$ is a cut-vertex of $G$ (no matter $N(s_4)\cap S_1= \emptyset$ or not). Altogether, we have $G[\{s_1,s_2,s_3,s_4, s_5\}]$ is vertex-critical by Theorem \ref{cc-c}. But clearly this is not true. (Refer to Figure \ref{123} (ii-A).)

\smallskip
\textbf{Case B.} At most one of $s_1$, $s_2$ and $s_3$, say $s_1$, has degree $2$ in $G$.

Then $d(s_2),d(s_3)\geq3$. Since $s_1\in N(s_5)$,
by Lemma \ref{y-t} (a), at least one of $N(s_2)\cap S_1=\emptyset$ and $N(s_3)\cap S_1=\emptyset$, say $N(s_2)\cap S_1=\emptyset$, holds. So $N(s_2)\subseteq S_3\cup S_4\cup S_5$, and thus $d(s_2)=3$. This implies that $N(s_2)\cap S_4\neq\emptyset$ and $N(s_2)\cap S_3\neq\emptyset$. From the former, we get $N(s_2)\cap S_4=\{s_4\}$ while by the latter we can let $N(s_2)\cap S_3=\{r_3\}$. ($r_3=s_3$ is possible.) Since $s_3\in N(s_5)$, we get

\vspace{-5mm}
\begin{equation}
N(s_1)\cap S_3=\emptyset
\label{13e}
\end{equation}

\vspace{-1mm}
\noindent by Lemma \ref{y-t} (a). There are two subcases.

When $N(s_1)\cap S_2=\emptyset$, we have $N(s_1)\cap S_4\neq\emptyset$ since $d(s_1)\geq2$. By (\ref{123-4}), we have $N(s_1)\cap S_4=\{s_4\}$. So $N(s_1)=\{s_4,s_5\}$. Thus by Lemma \ref{y-t} (b), we have $N(s_4)\cap S_5=\emptyset$. Since $r_3\in N(s_2)$, we get $N(s_4)\cap S_3=\emptyset$ by Lemma \ref{y-t} (a), and so $N(s_4)=\{s_1,s_2\}$. If $r_3=s_3$, then $s_3$ is a cut-vertex of $G$, and so $G[\{s_1,s_2,s_3,s_4,s_5\}]$ is vertex-critical, which is not true. If $r_3\neq s_3$, then $\{r_3,s_3\}$ is a vertex-cut of $G$. (Refer to Figure \ref{1234} (ii-B1).) By Observation \ref{cc-s} and Theorem \ref{cc-c}$^\prime$,  $G[\{s_1,s_2,s_3,s_4,s_5,r_3\}]$ is vertex-critical, which is also not true.

When $N(s_1)\cap S_2\neq\emptyset$, by (\ref{13e}) and Lemma \ref{y-t} (b), we have $N(s_1)\cap S_4\neq\emptyset$, which implies that $N(s_1)\cap S_4=\{s_4\}$. Since $s_2\in N(s_5)$, we have $N(s_3)\cap S_2=\emptyset$ by  Lemma \ref{y-t} (a). So~$d(s_3)=3$, and thus $N(s_3)\cap S_1\neq\emptyset$ and $N(s_3)\cap S_4\neq\emptyset$. By (\ref{123-4}), we have $N(s_3)\cap S_4=\{s_4\}$. (Refer to Figure \ref{1234} (ii-B2).) Now, we have $r_3\in N(s_2)$, $N(s_4)\cap S_3=\{s_3\}$ and $N(s_5)\cap S_3=\{s_3\}$. However, according to Lemma \ref{y-t} (a), this is impossible.

\vspace{4mm}
\begin{figure}[h]
\centering
\begin{minipage}[t]{0.95\linewidth}
\centering
\includegraphics[width=0.95\textwidth]{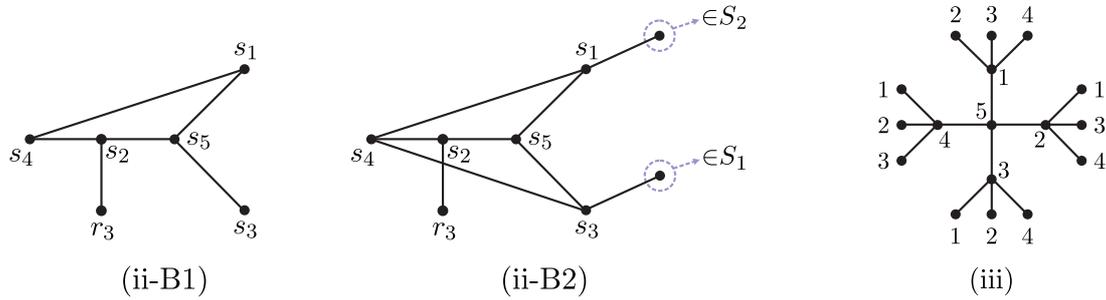}
\vspace{-2mm} \caption{Illustration for the proofs of Claim 2-B and Claim 3.}\label{1234}
\end{minipage}
\end{figure}

\smallskip
\emph{\textbf{Claim 3.}}~ $d(g)\neq4$ for every $g\in V(G)$.

Without loss of generality, suppose to the contrary that there exists some $s_5\in S_5$ such that $N(s_5)=\{s_1,s_2,s_3,s_4\}$, where $s_i\in S_i$, $i=1,2,3,4$. For every $1\leq i\leq 4$, by Lemma \ref{y-t} (b) and Claim 2, we have $d(s_i)\neq2$ and $d(s_i)\neq3$, which implies that $d(s_i)=4$. (Refer to Figure \ref{1234} (iii).) However, by Lemma \ref{y-t} (a), this is impossible.

\smallskip
By Claims 2 and 3, we get that $d_H(g)=2$ for every $g\in V(G)$,  which implies that $G$ is a cycle, a contradiction. The necessity follows, too.
\end{proof}
%

\vspace{-1mm}
\section{Conclusion}
\vspace{-1mm}
In \cite{Zhao}, the authors got  the following Proposition \ref{44-4}, which tells us that\hspace{1pt} $\mathscr{C}_4=\{C_4\}$,\hspace{1pt}  where \hspace{1pt}$\mathscr{C}_4$

\noindent was
defined in the last paragraph of Section 1. It is easy to see that the circulant graph $C_{12}\langle1,5\rangle$, the vertex coalescence $C_4\cdot C_4$ and  the Harary graph $H_{4,6}$ (see Figure \ref{HCE})  belong to $\mathscr{C}_6$. To compare Proposition \ref{44-4}, we want to know whether $\mathscr{C}_6$ is a finite set. So we present Problem \ref{66-6}.

\vspace{4mm}
\begin{figure}[h]
\centering
\begin{minipage}[t]{0.81\linewidth}
\centering
\includegraphics[width=0.82\textwidth]{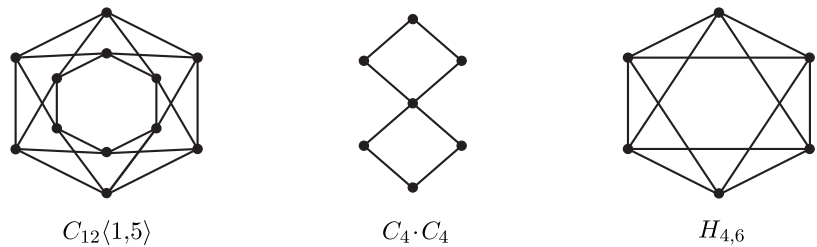}
\vspace{-2mm} \caption{Three elements of $\mathscr{C}_6$.}\label{HCE}
\end{minipage}
\end{figure}

\vspace{-1mm}
\begin{pro}\label{44-4}\emph{\cite{Zhao}}
Let $H$ be a connected graph. Then $V(H)$ can be partitioned into $4$   strong critical vertex-sets if and only if $H\cong C_4$.
\end{pro}

\begin{prob}\label{66-6}
Give a constructive characterization of the connected graphs $G$ such that $V(G)$ can be partitioned into $6$   strong critical vertex-sets of $G$.
\end{prob}

It is known that each   graph has a degree sequence, but a given sequence may be not a degree sequence of any  simple graph. For instance, the sequence (7,6,5,4,3,3,2) can not become a degree sequence of a simple graph (see \cite{Bondy}, Ex.~1.5.6). If $V(G)=S_1\cup S_2\cup\cdots\cup S_l$ is a strong critical vertex-sets partition of a  graph $G$, then we call the sequence $(|S_1|, |S_2|,\ldots, |S_l|)$ as a \emph{strong critical vertex-sets sequence} of $G$. It is noteworthy that even a connected graph may own different strong  critical vertex-sets sequences. For example, both (3,2,2,1,1,1,1,1,1) and (2,2,2,2,1,1,1,1,1) are strong  critical vertex-sets sequences of the graph depicted in Figure \ref{5-5-5}. Also, for connected graphs, it follows from Theorem \ref{main} that the strong critical vertex-sets sequence (1,1,1,1) exists but (1,1,1,1,2) does not exist.

\vspace{4pt}
\begin{figure}[h]
\centering
\begin{minipage}[t]{0.79\linewidth}
\centering
\includegraphics[width=0.82\textwidth]{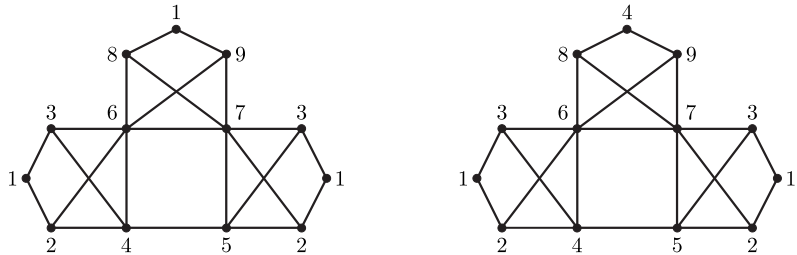}
\vspace{-2mm} \caption{A graph with more than one strong critical vertex-sets sequences.}\label{5-5-5}
\end{minipage}
\end{figure}

\vspace{-2mm}
\begin{prob}\label{qq-q}
What kinds of strong  critical vertex-sets sequences do exist\hspace{2pt}$?$ Or to be concrete about it, if $(|S_1|, |S_2|,\ldots, |S_l|)$ is a strong  critical vertex-sets sequence of a connected graph $G$, then what are the relations of $|S_1|, |S_2|,\ldots, |S_l|$, $l$ and $\gamma(G)$\hspace{1pt}$?$
\end{prob}

\vspace{1mm}
\section{Acknowledgement}

\vspace{-1mm}
Thank Professor D. F. Rall for his suggestion on the definition of strong critical vertex-set,  and e-mailing the article (ref.~\cite{Phillips}) to us\hspace{1pt}!

\end{document}